\documentclass[a4paper, 9pt, twosides]{amsart}

\usepackage{lmodern}
\usepackage{amssymb}
\usepackage{amsthm}
\usepackage{amsaddr}
\usepackage{mathtools}
\usepackage{MnSymbol}
\usepackage{graphicx}
\usepackage{caption}
\usepackage{subcaption}
\usepackage{url}

\usepackage{enumitem}
\setitemize{nolistsep}
\setenumerate{nolistsep}
\usepackage{todonotes}
\usepackage{algorithm,algorithmic}

\allowdisplaybreaks

\newtheorem{theorem}{Theorem}
\newtheorem{proposition}{Proposition}

\theoremstyle{definition}
\newtheorem{defn}{Definition}
\newtheorem{example}{Example}

\theoremstyle{remark}
\newtheorem{remark}{Remark}

\theoremstyle{assumption}

\theoremstyle{fact}

\theoremstyle{claim}

\theoremstyle{prob}
\newtheorem{prob}{Problem}

\theoremstyle{algo}
\newtheorem{algo}{Algorithm}

\theoremstyle{experiment}

\numberwithin{equation}{section}

\newcommand{\norm}[1]{\left\lVert{#1}\right\rVert}

\newcommand{\pmat}[1]{\begin{pmatrix}#1\end{pmatrix}}

\newcommand{\R}{\mathbb{R}}
\newcommand{\N}{\mathbb{N}}
\renewcommand{\P}{\mathcal{N}}

\newcommand{\X}{\mathcal{X}}
\newcommand{\Z}{\mathbb{Z}}
\newcommand{\Nset}{\mathcal{N}}
\newcommand{\E}{\mathcal{E}}
\renewcommand{\H}{\mathcal{H}}
\newcommand{\Y}{\mathcal{Y}}
\newcommand{\U}{\mathcal{U}}
\newcommand{\C}{\mathcal{C}}
\newcommand{\W}{\mathcal{W}}
\newcommand{\T}{\mathcal{T}}

\DeclareMathOperator{\minimize}{minimize}

\allowdisplaybreaks

\title[]{On maximum hands-off restricted hybrid control for discrete-time switched linear systems}
\author{Darsana U and Atreyee Kundu}
\address{Department of Electrical Engineering,\\Indian Institute of Technology Kharagpur\\West Bengal - 721302, India,\\ E-mail: darsana\_udayakumar@kgpian.iitkgp.ac.in, atreyee@ee.iitkgp.ac.in}


\date{\today}

\begin{document}

	\begin{abstract}
         This paper deals with design of maximum hands-off hybrid control sequences for discrete-time switched linear systems. It is a sparsest combination of a discrete control sequence (i.e. the switching sequence) and a continuous control sequence, both satisfying pre-specified restrictions on the admissible actions, that steers a given initial state of the switched system to the origin of the state-space in a pre-specified duration of time. Given the subsystems dynamics, the sets of admissible continuous and discrete control, the initial state and the time horizon, we present a new algorithm that, under certain conditions on the subsystems dynamics and the admissible control, designs maximum hands-off hybrid control sequences for the resulting switched system. The key apparatuses for our analysis are graph theory and linear algebra. Numerical examples are presented to demonstrate our results.
    \end{abstract}

    \maketitle
\section{Introduction}
\label{s:intro}
    \emph{Maximum hands-off} is an optimal control paradigm where non-zero control efforts over a finite horizon of time is minimized while not compromising on a desired control objective. Such control techniques are a natural choice in the context of optimal control under resource constraints and find wide applications in energy-aware control \cite[Section 8.2]{Nagahara}, minimum fuel control \cite[Section 1.4.1]{Nagahara}, etc. Since the inception of maximum hands-off control paradigm in \cite{Nagahara2016a,Nagahara2016b}, researchers have restricted their attention to time-invariant (non)linear systems, and primarily focussed on theoretical development towards designing a maximum hands-off control sequence, under which a system achieves a certain objective, in a numerically tractable manner. In particular, the theory of maximum hands-off control is so far explored for continuous-time and discrete-time deterministic control-affine systems \cite{Nagahara2016a}, discrete-time linear systems \cite{Nagahara2016b}, a class of continuous-time nonlinear systems \cite{Budhraja2021}, discrete-time stochastic linear systems \cite{Kishida2023}. Recently in \cite{abc} we extended the notion of maximum hands-off control to discrete-time switched linear systems. 
    
    In this paper we continue with our study of maximum hands-off hybrid control for discrete-time switched linear systems. Our focus is on the setting where the subsystems of the switched system operate in closed-loop, i.e., in addition to the switching sequence and the initial condition, the state evolution is governed by a control input.  Since the switching sequence maps to a finite number of subsystems and the control inputs are real numbers of appropriate dimension, they are referred to as discrete control and continuous control, respectively, and as hybrid control together. We also assume that the set of admissible switches between the subsystems and the set of admissible continuous control actions on the subsystems are pre-specified. A restriction on discrete control is natural to systems where switches between certain subsystems are prohibited, e.g., automobile  gear switching, while a restriction on continuous control is natural to systems where actuator saturations and noise are prevalent, e.g., aerospace systems. We call a hybrid control sequence as maximum hands-off if it satisfies the given restrictions on continuous and discrete control and is a sparsest in the sense that the sum of (a) the number of switches from one subsystem to another, i.e., changes in the discrete control in consecutive time instants and (b) the number of non-zero continuous control, is minimum among all hybrid control sequences that satisfy the given restrictions and steer a given initial state of the switched system to the origin of the state-space in a given time horizon. 
       It is clear that a maximum hands-off hybrid control can be expressed as solution to a minimization problem whose (i) decision variable is a hybrid control sequence that obeys pre-specified restrictions on continuous and discrete control and achieves the above mentioned control objective, and (ii) the objective function is the sum of individual sparsity of two vectors --- the first vector contains the difference between every two consecutive elements of the discrete control sequence and the second vector is the continuous control sequence. Notice that this optimization problem is non-convex --- in particular, the domain of discrete control is non-convex and sparsity of a vector (commonly measured by \(\ell_0\)-norm) is a non-convex function. As a result, numerical design of maximum hands-off hybrid control sequences is not a straightforward task. In this paper we are interested in the algorithmic design of maximum hands-off hybrid control sequences for discrete-time switched linear systems.
   
      We studied the above problem in the absence of pre-specified restrictions on discrete and continuous control in \cite{abc}, which to the best of our knowledge, is the only work on maximum hands-off hybrid control for discrete-time switched linear systems available in the literature. An algorithm to design sparse hybrid control sequences by employing a mixed-integer convex optimization problem was presented. The optimization problem computes a sparsest hybrid control sequence employing \(\ell_1\)-approximation (i.e., a convex relaxation of \(\ell_0\)-norm). It was shown that the \(\ell_0\)-sparsest and \(\ell_1\)-sparsest quantities coincide under satisfaction of certain conditions (viz., restricted isometry property and range space property) by certain matrices constructed using the subsystems dynamics, the given initial condition and the given time horizon. However, in practice, numerical verification of these sufficient conditions are often difficult. Indeed, the class of deterministic matrices that satisfy restricted isometry property is small \cite{Bourgain2011} and often not representative of the dynamics of a practical switched system. In this paper we present a new algorithm that designs maximum hands-off hybrid control sequences for discrete-time switched linear systems that admit a certain abstraction of the state-space. 
      
    We consider the subsystems dynamics, the sets of admissible continuous and discrete control, the initial state and the time horizon to be given, and take the following route towards the design of a maximum hands-off hybrid control: 
    \begin{itemize}[label=\(\circ\),leftmargin=*]
        \item Firstly, we construct an abstraction of the state-space as a set of mutually disjoint sets of states (representing different regions of the state-space) that together span the entire state-space and admit well-defined transitions under admissible switches and continuous control. 
        \item Secondly, we associate a labelled directed graph to a switched system and call it a transition graph. The vertices of this graph are labelled with the elements of the state-space abstraction, and the edges of this graph represent transitions between the different regions of the state-space under switches and continuous control. The edges are labelled to keep track of the corresponding switches and continuous control causing a transition. 
        \item Finally, we show that a maximum hands-off hybrid control sequence corresponds to a minimum weight walk of a certain length that starts and ends at certain specific vertices on the transition graph of a switched system in the sense that the maximum hands-off discrete and continuous control sequences can be constructed by employing the relevant information on the weights of the edges that appear in the walk. 
    \end{itemize}
    
    With our techniques, maximum hands-off hybrid control sequences can be designed by employing off-the-shelf graph algorithms that find minimum weight walks on a directed graph. We also present a set of sufficient conditions on the subsystems dynamics and the admissible control under which the state-space abstraction, which is a cornerstone of our results, is admitted by a switched system. To the best of our knowledge, this is the first instance in the literature where the design of maximum hands-off hybrid control sequences for discrete-time switched linear systems under pre-specified restrictions on continuous and discrete control, is addressed using graph theory. We also demonstrate through numerical examples that our results cater to switched systems where restricted isometry (resp., range space property) based maximum hands-off guarantees, presented in our earlier work \cite{abc}, cannot be inferred.
    
    The remainder of this paper is organized as follows: In Section \ref{s:prob-stat} we formulate the problem under consideration. A set of preliminaries relevant to our results is stated in Section \ref{s:prelims}. Our results appear in Section \ref{s:mainres}. We present numerical examples in Section \ref{s:num-ex}. We conclude in Section \ref{s:concln} with a brief discussion on future research directions. Proofs of our results appear in Section \ref{s:all-proofs}.
    
    {\bf Notation}.  \(\R\) is the set of real numbers, \(\N\) is the set of natural numbers, and \(\Z\) is the set of integers. For \(z\in\R^n\), \(\norm{z}_0\) denotes its \(\ell_0\)-norm, i.e., the number of non-zero elements in \(z\). We let \(0_n\) denote an \(n\)-dimensional zero vector.

\section{Problem statement}
\label{s:prob-stat}
     We consider a discrete-time switched linear system 
    \begin{align}
    \label{e:swsys}
        x(t+1) = A_{\nu(t)}x(t)+b_{\nu(t)}\mu(t),\:\:x(0)=x_0,\:t=0,1,\ldots,T-1,
    \end{align}
    where \(x(t)\in\X\subseteq\R^d\) is the vector of states at time \(t\), and \(\mu(t)\in\R\) and \(\nu(t)\in\Nset:=\{1,2,\ldots,N\}\) are the continuous control and discrete control at time \(t\), respectively. 
    
    The pair \((\nu(t),\mu(t))\) is called a \emph{hybrid control} at time \(t\). Let
       \( \mu_T = \pmat{\mu(0)\\\mu(1)\\\cdots\\\mu(T-1)}\) 
       and \(\nu_T = \pmat{\nu(0)\\\nu(1)\\\cdots\\\nu(T-1)}\)
    denote the continuous control and discrete control sequences, respectively. We will call the pair \((\nu_T,\mu_T)\) a \emph{hybrid control sequence}. 
    
    Let the set of admissible switches between the subsystems, 
    \[
        \E(\Nset):= \{(i,j)\:|\:\:i\:\:\text{to}\:\:j\:\:\text{is allowed},\:i,j\in\P\}\subseteq\P\times\P
    \] 
    and the set of admissible continuous control inputs, \(\U\subseteq\R\), be given. We will let \(\P_s\subseteq\P\) denote the set of all \(i\in\P\) such that \((i,i)\in\E(\P)\).
    
    \begin{defn}
    \label{d:adm_cont}
        The hybrid control sequence, \((\nu_T,\mu_T)\), is \emph{admissible}, if the following holds:
        \begin{align}
            \label{e:adm-dis} &(\nu(t),\nu(t+1))\in\E(\Nset),\:\:t=0,1,\ldots,T-2,
            \intertext{and}
            \label{e:adm-con} &\mu(t)\in\U,\:\:t=0,1,\ldots,T-1.
        \end{align}
    \end{defn}
     Let \(\H\) denote the set of all admissible hybrid control sequences.
    
    In the sequel we will assume that the subsystems dynamics, \((A_i,b_i)\in\R^{d\times d}\times\R^d\), \(i\in\Nset\), the set of admissible switches between the subsystems, \(\E(\Nset)\), the set of admissible control inputs, \(\U\), and the time horizon, \(T\in\N\), are pre-specified. 
    
    Given \(\xi\in\R^d\), we define
    \begin{align*}
         \C_{\xi,T} := \Bigl\{(\nu_T,\mu_T)\in\H\:\big|\:x(0)&=\xi,\\
         x(t+1) &= A_{\nu(t)}x(t)+b_{\nu(t)}\mu(t),
       \:t=0,1,\ldots,T-1,\\
        x(T)&=0_d\Bigr\}.
    \end{align*}
    It follows that \(\C_{\xi,T}\subseteq\H\) is the set of all admissible hybrid control sequences, \((\nu_T,\mu_T)\), that steer the state of the switched system \eqref{e:swsys} from \(x(0)=\xi\) to the origin in \(T\) time instants, i.e., \(x(T)=0_d\). We are interested in the sparsest element of the set \(\C_{\xi,T}\). 
    
    Let us define \(\Delta:\Nset^T\to\Z^{T-1}\) as 
    \begin{align*}
             \nu_T\mapsto\Delta(\nu_T):= \pmat{\nu(1)-\nu(0) \\ \nu(2)-\nu(1) \\ \cdots \\ \nu(T-1)-\nu(T-2)}.
    \end{align*}
    \begin{defn}
    \label{d:hybrid}
    \rm{
        An admissible hybrid control sequence, \(({\nu}^*_T,{\mu}^*_T)\in\C_{\xi,T}\), is called a \emph{maximum hands-off hybrid control sequence} for the switched system \eqref{e:swsys}, if it is a solution to the following optimization problem:
        \begin{align}
        \label{e:optprob1}
            \underset{{(\nu_T,\mu_T)\in\C_{\xi,T}}}\minimize\:\:\norm{\Delta(\nu_T)}_0+\norm{\mu_T}_0.
        \end{align}
    }
    \end{defn}
    
           A maximum hands-off hybrid control sequence minimizes the total number of time instants when an admissible continuous control or, an admissible discrete control or, both are to be applied in order to steer the initial state \(x(0)=\xi\) of the switched system \eqref{e:swsys} to the origin, \(0_d\), in \(T\) time instants. On the one hand, a minimization of \(\norm{\mu_T}_0\) gives a continuous control sequence, \(\tilde{\mu}_T\), with elements satisfying \(\tilde{\mu}(t)\in\U\), \(t=0,1,\ldots,T-1\) that minimizes the total number of time instants when a continuous control is to be applied. On the other hand, a minimization of \(\norm{\Delta(\nu_T)}_0\) gives a discrete control sequence, \(\tilde{\nu}_T\), with elements satisfying \((\tilde{\nu}(t),\tilde{\nu}(t+1))\in\E(\Nset)\), \(t=0,1,\ldots,T-2\) that minimizes the total number of time instants when the discrete control changes its value. A maximum hands-off hybrid control sequence is, however, not necessarily \((\tilde{\nu}_T,\tilde{\mu}_T)\). Indeed, individual minimization of \(\norm{\mu_T}_0\) and \(\norm{\Delta(\nu_T)}_0\) does not necessarily ensure that \((\tilde{\nu}_T,\tilde{\mu}_T)\) steers \(x(0)=\xi\) to \(x(T)=0_d\). We aim for a hybrid control sequence, \((\overline{\nu}_T,\overline{\mu}_T)\) whose discrete and continuous control components are admissible and together achieve the desired control objective while maximizing the total number of time instants when a hybrid controller needs not take an action. 
       
    \begin{remark}
    \label{rem:non-convex}
    \rm{
        The optimization problem \eqref{e:optprob1} is non-convex due to the non-convexity of (a) \(\ell_0\)-norm \cite[Chapters 2 and 3]{Nagahara} and (b) the domain of the discrete control, \(\Nset\). It is well-known that solving such problems is numerically hard. A search for algorithms to solve the optimization problem \eqref{e:optprob1} is, therefore, of interest. 
    }
    \end{remark}

   In this paper we will solve the following:
    \begin{prob}
    \label{prob:main}
        Given the subsystems dynamics, \((A_i,b_i)\in\R^{d\times d}\times\R^{d}\), \(i\in\Nset\), the set of admissible switches between the subsystems, \(\E(\Nset)\subseteq\P\times\P\), the set of admissible continuous control, \(\U\subseteq\R\), an initial condition, \(x(0)=\xi\in\R^d\), and a time horizon, \(T\in\N\), design (algorithmically) a maximum hands-off hybrid control sequence, \(({\nu}^*_T,{\mu}^*_T)\in\C_{\xi,T}\), for the switched system \eqref{e:swsys}.
    \end{prob}
      
    We present a solution to Problem \ref{prob:main} employing graph-theoretic entities in Section \ref{s:mainres}. Prior to that we catalog a set of preliminaries. 
\section{Preliminaries}
\label{s:prelims}
 \begin{defn}
    \label{d:partition}
    \rm{
        The sets of states, \(\bigl(\X_j\bigr)_{j=0}^{n}\), is called a \emph{state-space abstraction} of the switched system \eqref{e:swsys}, if the following conditions are satisfied: 
        \begin{enumerate}[label = \roman*),leftmargin=*]
            \item \(\X_0 = \{0_d\}\),
            \item \(\X_j\cap\X_k = \emptyset\) for all \(j,k\in\{0,1,\ldots,n\}\), \(j\neq k\),
            \item \(\displaystyle{\bigcup_{j=0}^{n}}\X_j = \X\),
            \item for each \(\X_j\in\{\X_1,\X_2,\ldots,\X_n\}\) and \(i\in\P\), there exists \(\X_k\in\{\X_0,\X_1,\ldots,\X_n\}\) such that for all \(x\in\X_j\), \(y=A_ix+B_i\mu\in\X_k\) with \(\mu=0\), and
            \item for each \(\X_j\in\{\X_1,\X_2,\ldots,\X_n\}\) and \(i\in\P\), there exists \(\X_{\ell}\in\{\X_0,\X_1,\ldots,\X_n\}\) such that for all \(x\in\X_j\), \(y=A_ix+B_i\mu\in\X_{\ell}\) with \(\mu=\mu(x)\in\U\setminus\{0\}\).
        \end{enumerate}
    }
    \end{defn}
    
    The state-space abstraction, described in Definition \ref{d:partition}, partitions the state-space into a finite number of regions that are mutually disjoint, together span the entire state-space, and the transitions between the regions under the subsystems dynamics are well-defined. Notice that the individual sets in the partition, \(\X_j\), \(j=1,2,\ldots,n\) are allowed to contain infinite elements. Clearly, whether such a partition exists, depends on the subsystems dynamics, \((A_i,b_i)\), \(i\in\Nset\) and the set of admissible continuous control, \(\U\). 
    We shall employ this state-space abstraction towards designing maximum hands-off hybrid control sequences for the switched system \eqref{e:swsys}. Our results will, therefore, be applicable to discrete-time switched linear systems that admit the state-space abstractions under consideration. In Section \ref{s:mainres} we will provide sufficient conditions on the subsystems dynamics, \((A_i,b_i)\), \(i\in\Nset\) and the set of admissible continuous control, \(\U\), under which a switched system \eqref{e:swsys} will admit a state-space abstraction.
      
    \begin{defn}
    \label{d:graph}
        The \emph{transition graph} of the switched system \eqref{e:swsys} is the labelled directed graph, \(G(V,E)\), that consists of the following components:
        \begin{itemize}[label=\(\circ\),leftmargin=*]
            \item a set of vertices, \(V = \{v_0,v_1,\ldots,v_n\}\), where each element is labelled as \(\ell(v_p) = \X_p\), \(p=0,1,\ldots,n\), and
            \item a set of edges, \(E = E_1\cup E_2\), where 
                \(
                    E_1 = \{(v_p,v_q)\:|\:\ell(v_p)=\X_j,\ell(v_q)=\X_k\:\text{such that there}\\\text{ exists}\:\)
               \( i\in\P\:\text{satisfying}\:A_ix+B_i\mu\in\X_k\:\text{for all}\:x\in\X_j\:\text{and}\:\mu=0 \},
                \)
                and
                \(
                    E_2 = \{(v_p,v_q)\:|\:\ell(v_p)=\X_j,\ell(v_q)=\X_{\ell}\:\text{such that}\:\text{ there exists}\)
                    \(\:
                i\in\P\:\text{satisfying}\:A_ix+B_i\mu\in\X_{\ell}\:\: \text{for all}\:x\in\X_j\:
                \text{and}\text{ some}\\\:\mu=\mu(x)\in\U\setminus\{0\} .
                \)
            The edges \((v,v')\in E_1\) are labelled as \(\ell(v,v')=(i,0)\) and
            the edges \((v,v')\in E_2\) are labelled as \(\ell(v,v')=(i,\mu(x))\).
        \end{itemize}
    \end{defn}
        
        \begin{remark}
        \label{rem:multi-graph}
        \rm{
        Notice that \(G(V,E)\) is a multi-graph possibly in the sense that between two fixed vertices, multiple edges with different labels may exist. For simplicity of exposition, we will continue with the word \emph{graph} unless an explicit mention of the word \emph{multi-graph} is necessary.
        }
        \end{remark}
        
        Fix \((v,v')\in E\). We will use \(\alpha(v,v')\) and \(\beta(v,v')\) to refer to the first element of \(\ell(v,v')\) and the second element of \(\ell(v,v')\), respectively.

    In Definition \ref{d:graph} we associate a labelled directed graph, \(G(V,{E})\), to the switched system \eqref{e:swsys}. The vertices of this graph are the elements in the state-space abstraction of the switched system under consideration. In other words, the vertices represent different regions of the state-space. The edges represent transitions between these regions. The labels on the edges carry information about the cause of the transitions in terms of the discrete and continuous control. We shall connect maximum hands-off hybrid control sequences with entities on the transition graph of the switched system \eqref{e:swsys}.
    
    \begin{defn}
    \label{d:walk}
    \rm{
        A \emph{T-walk} on \(G(V,E)\) is a walk \(W_T = \overline{v}_0,(\overline{v}_0,\overline{v}_1)\),
        \(\overline{v}_1,(\overline{v}_1,\overline{v}_2),\overline{v}_2,\ldots,\overline{v}_{r-1}\),\\\((\overline{v}_{r-1},\overline{v}_{r}),\overline{v}_r\) that satisfies:
        \begin{enumerate}[label=\roman*),leftmargin=*]
            \item \((\alpha(\overline{v}_m,\overline{v}_{m+1}),\alpha(\overline{v}_{m+1},\overline{v}_{m+2}))\in\E(\Nset)\), \(m=0,1,\ldots,r-2\), and
            \item the length of \(W_T\) is at most \(T\), i.e., \(r\leq T\).
        \end{enumerate}
    }
    \end{defn}
    
    \begin{defn}
    \label{d:weight}
    \rm{
        The \emph{weight} of a \(T\)-walk, \(W_T\), on \(G(V,{E})\) is computed as
        \[
            w(W_T) = \sum_{m=1}^{r-1}\overline{\alpha}(\overline{v}_m,\overline{v}_{m+1})+ \sum_{m=0}^{r-1}{\overline{\beta}}(\overline{v}_m,\overline{v}_{m+1}),
        \]
        where 
        \begin{align*}
            \overline{\alpha}(\overline{v}_m,\overline{v}_{m+1}) &=
            \begin{cases}
                1,\:\:\text{if}\:\:\alpha(\overline{v}_m,\overline{v}_{m+1})\neq\alpha(\overline{v}_{m-1},\overline{v}_{m}),\\
                0,\:\:\text{if}\:\:\alpha(\overline{v}_m,\overline{v}_{m+1})=\alpha(\overline{v}_{m-1},\overline{v}_{m}),
            \end{cases}
            \intertext{and}
              \overline{\beta}(\overline{v}_m,\overline{v}_{m+1}) &=
            \begin{cases}
                1,\:\:\text{if}\:\:\beta(\overline{v}_m,\overline{v}_{m+1})\neq 0,\\
                0,\:\:\text{if}\:\:\beta(\overline{v}_m,\overline{v}_{m+1})=0,
            \end{cases}
        \end{align*}
        \(m=1,2,\ldots,r-1\).
        }
    \end{defn}
    
     Let \(\W_{\T}(\X,\Y)\) be the set of all \(T\)-walks, \(W_T = \overline{v}_0,(\overline{v}_0,\overline{v}_1)\),
        \(\overline{v}_1,(\overline{v}_1,\overline{v}_2),\overline{v}_2,\ldots,\overline{v}_{r-1}\),\\\((\overline{v}_{r-1},\overline{v}_{r}),\overline{v}_r\), on \(G(V,{E})\), that satisfy \(\ell(\overline{v}_0) = \X\) and \(\ell(\overline{v}_r) = \Y\).
        
        \begin{defn}
        \label{d:hands-off walk}
        \rm{
            Let \(\xi\in\R^d\) be given. Suppose that \(\xi\in\X_s\), \(s\in\{0,1,\ldots,n\}\). A \emph{\(\xi\)-hands-off} walk on \(G(V,{E})\) is a \(T\)-walk, \(W^*_T = \overline{v}_0,(\overline{v}_0,\overline{v}_1)\),
        \(\overline{v}_1,(\overline{v}_1,\overline{v}_2),\overline{v}_2,\ldots,\overline{v}_{r-1},(\overline{v}_{r-1},\overline{v}_{r}),\overline{v}_r\in\)\\\(\W_\T(\X_s,\X_0)\), that satisfies 
        \(\ell(\overline{v}_0)=\X_s\), \(\ell(\overline{v}_r)=\X_0\) and \(w(W^*_T)\leq w(W_T)\) for all \(W_T\in\W_\T(\X_s,\X_0)\).
            }
        \end{defn}
    
    We will employ hands-off walks on \(G(V,{E})\) for the design of maximum hands-off control for the switched system \eqref{e:swsys}. We are now in a position to present our solution to Problem \ref{prob:main}.
\section{Main results}
\label{s:mainres}
    Our solution to Problem \ref{prob:main} is the following:
    
    \begin{theorem}
    \label{t:mainres}
        Consider the switched system \eqref{e:swsys}. Suppose that the subsystems dynamics, \((A_i,b_i)\), \(i\in\P\), the set of admissible switches, \(\E(\Nset)\), the set of admissible continuous control, \(\U\), the initial state, \(\xi\), and the time horizon, \(T\), are given, and that the switched system \eqref{e:swsys} admits a state-space abstraction as described in Definition \ref{d:partition}. Let \(\W_{\T}(\X_s,\X_0)\neq\emptyset\). Then a hybrid control sequence obtained from Algorithm \ref{algo:design} is a maximum hands-off hybrid control sequence, \((\overline{\nu}^*_T,\overline{\mu}^*_T)\in\C_{\xi,T}\).
    \end{theorem}
    
     \begin{algo}
    \label{algo:design}
    \rm{
        {\bf Computation of a maximum hands-off hybrid control sequence for the switched system \eqref{e:swsys}}
        \begin{enumerate}[label=Step \Roman*.]
        	 \item Let \(W^*_T = \overline{v}_0,(\overline{v}_0,\overline{v}_1)\),
        \(\overline{v}_1,(\overline{v}_1,\overline{v}_2),\overline{v}_2,\ldots,\overline{v}_{r-1}\),\((\overline{v}_{r-1},\overline{v}_{r}),\overline{v}_r\in \W_{\T}(\X_s,\X_0)\) be a \(\xi-\)hands-off walk on \(G(V,{E})\).
            \item\label{step:algo_step2} Construct a hybrid control sequence, \((\overline{\nu}_T,\overline{\mu}_T)\), as follows:
            \begin{enumerate}[label= Case \Roman*.]
                \item If \(r=T\), then construct the discrete control sequence as: 
                \[
                    \overline{\nu}(t) = \alpha(\overline{v}_t,\overline{v}_{t+1}),\:\:t=0,1,\ldots,r-1,
                \]
                and construct the continuous control sequence as: 
                \[
                    \overline{\mu}(t)=\beta(\overline{v}_t,\overline{v}_{t+1}),\:\:t=0,1,\ldots,r-1.
                \]
                \item If \(r<T\), then construct the discrete control sequence as: 
                \[
                    \overline{\nu}(t) = \alpha(\overline{v}_t,\overline{v}_{t+1}),\:\:t=0,1,\ldots,r-1
                \] 
                and 
                \begin{align*}
                    \overline{\nu}(t) = 
                    \begin{cases}
                        \overline{\nu}(t-1),\:\:&\text{if}\:(\overline{\nu}(t-1),\overline{\nu}(t-1))\in\E(\Nset),\\
                        i\in\P_s,\:\:&\text{if}\:(\overline{\nu}(t-1),\overline{\nu}(t-1))\notin\E(\Nset),\:i\in\P_s\:\text{and}\:(\overline{\nu}(t-1),i)\in\E(\Nset),\\
                        j\in\P,\:\:&\text{if}\:(\overline{\nu}(t-1),\overline{\nu}(t-1))\notin\E(\Nset),\:(\overline{\nu}(t-1),i)\notin\E(\Nset)\:\text{for all}\:i\in\P_s\:\\
                        &\qquad\qquad\text{and}\:(\overline{\nu}(t-1),j)\in\E(\Nset),
                    \end{cases}
                \end{align*}
                \(t=r,r+1,\ldots,T-1\), and construct the continuous control sequence as: 
                \[
                    \overline{\mu}(t)=\beta(\overline{v}_t,\overline{v}_{t+1}),\:\:t=0,1,\ldots,r-1                
                \]
                and
                \[
                    \overline{\mu}(t)=0,\: t=r,r+1,\ldots,T-1.
                \]
            \end{enumerate}
        \end{enumerate}
        }
    \end{algo}
      
    Algorithm \ref{algo:design} designs a hybrid control sequence by employing a \(\xi\)-hands-off walk, \(W^*_T = \overline{v}_0,(\overline{v}_0,\overline{v}_1),\overline{v}_1,\ldots,\overline{v}_{T-1}\),\((\overline{v}_{r-1},\overline{v}_r)\),\(\overline{v}_r\in \W_{\T}(\X_s,\X_0)\) on \(G(V,{E})\). In particular, the discrete control, \(\nu(t)\), \(t=0,1,\ldots,r-1\) and the continuous control, \(\mu(t)\), \(t=0,1,\ldots,r-1\), are chosen as the subsystems indices and continuous control that appear as the first and the second element on the labels of the edges that appear in \(W^*_T\), respectively. In case of \(r<T\), the discrete control, \(\nu(t)\) and the continuous control, \(\mu(t)\), \(t=r,r+1,\ldots,T-1\) are chosen obeying the elements of \(\E(\Nset)\) ensuring minimum possible addition to sparsity and \(\mu=0\), respectively. This is no loss of generality as \(A_i0_d+b_i0=0_d\) for all \(i\in\P\). Theorem \ref{t:mainres} asserts that for a switched system \eqref{e:swsys} that admits a state-space abstraction, described in Definition \ref{d:partition}, such that the set of \(\xi\)- hands-off walks on the transition graph is non-empty, a hybrid control sequence obtained from Algorithm \ref{algo:design} is a maximum hands-off hybrid control sequence. We present a proof of Theorem \ref{t:mainres} in Section \ref{s:all-proofs}.
    
    \begin{remark}
    \label{rem:extra-comp}
    \rm{
        Notice that in the construction of continuous control sequence in Algorithm \ref{algo:design}, whenever \(\beta(\overline{v}_t,\overline{v}_{t+1})\), \(t=0,1,\ldots,r-1\), is dependent on \(x\), an explicit computation of \(x(t)\), \(t\in\{1,2,\ldots,r-1\}\) is required under \(x(0)=\xi\) and the already chosen \(\nu(\tau)\) and \(\mu(\tau)\), \(\tau=0,1,\ldots,t-1\). If \(\beta(\overline{v}_t,\overline{v}_{t+1})\), \(t=0,1,\ldots,r-1\) is any \(\mu\in\U\setminus\{0\}\), then this computation can be saved.
    }
    \end{remark}

    \begin{remark}
    \label{rem:compa}
    \rm{
        In the existing work \cite{abc} maximum hands-off guarantee to sparse hybrid control sequences under unrestricted switches between the subsystems and unrestricted continuous control, was provided using restricted isometry property (resp., range space property) of certain matrices involving the subsystems dynamics, the initial state and the time horizon. In Theorem \ref{t:mainres} (resp., Algorithm \ref{algo:design}) we design maximum hands-off hybrid control sequences for discrete-time switched linear systems under pre-specified restrictions on (a) the switches between the subsystems and (b) the continuous control, and rely on admissibility of a state-space abstraction by the switched system under consideration. We will present numerical examples in Section \ref{s:num-ex} where the results of \cite{abc} do not hold but the results presented in this paper are applicable. However, in the absence of a mathematical relation between restricted isometry property (resp., range space property) and admissibility of a state-space abstraction, we are unable to compare the sizes of the classes of switched systems \eqref{e:swsys} that the results of \cite{abc} and of this paper cater to. 
    }
    \end{remark}
    
     Let us now consider that a state-space abstraction of the switched system \eqref{e:swsys} and its corresponding transition graph, \(G(V,{E})\), be given. Towards implementing Algorithm \ref{algo:design}, we require to find a \(\xi\)-hands-off walk on \(G(V,{E})\). The following algorithm can be employed for this purpose.
    \begin{algo}
    \label{algo:walk-design}
    \rm{
        {\bf Construction of a hands-off walk required for Algorithm \ref{algo:design}}
        \begin{enumerate}[label=Step \Roman*.]
            \item\label{alg2-step1} Find all \(T\)-walks \(W_T = \overline{v}_0,(\overline{v}_0,\overline{v}_1)\),
        \(\overline{v}_1,(\overline{v}_1,\overline{v}_2),\overline{v}_2,\ldots\),\(\overline{v}_{r-1}\),\((\overline{v}_{r-1},\overline{v}_{r}),\overline{v}_r\in \W_{\T}(\X_s,\X_0)\).
            \item Compute the weight, \(w(W_T)\), for all \(W_T\in \W_{\T}(\X_s,\X_0)\) and choose \(W^*_T\) to be one that yields a minimum value of \(w(W_T)\).
        \end{enumerate}
        }
    \end{algo}
    
    The problem of listing all walks of length less than equal to \(T\) between the two vertices \(u\) and \(v\), where \(\ell(u)=\X_s\) and \(\ell(v)=\X_0\) on a directed graph (employed in \ref{alg2-step1} of Algorithm \ref{algo:walk-design}) is a standard problem in graph algorithms and can be solved by using minor logical modifications in off-the-shelf graph algorithms, e.g., depth-first search \cite{Cormen_algo}, breadth-first search \cite{Cormen_algo}, etc.

    Our results so far rely on the admissibility of a state-space abstraction by the switched system \eqref{e:swsys}. Intuitively, if the subsystems dynamics, \((A_i,b_i)\), \(i\in\Nset\) and the set of admissible continuous control, \(\U\) together can map large portions of the state-space into fixed regions, then a finite set of sets of states abstracting the state-space can be found. The number of elements in the abstraction may, however, be large requiring us to deal with large directed graphs. An iterative trial-and-error search for the construction of the state-space abstraction can be conducted as follows: Let \(\X = \R^n\). Define \(\X_0 = \{0_d\}\) and \(\X_{j^{(k)}} = \{x\in\R^n\:|\:\text{only}\:\:k\:\:\text{elements of}\:\:x\:\:\text{are non-zero}\}\), \(j^{(k)} = 1,2,\ldots,{d\choose k}\), \(k=1,2,\ldots,d\). Notice that the components of \(\{\X_0,\X_1,\ldots,\X_n\}\), \(\displaystyle{n = \sum_{k=1}^{d}{d\choose k}}\) are mutually disjoint and their union spans the entire state-space. Now, using the subsystems dynamics, \((A_i,b_i)\), \(i\in\P\) and the set of admissible continuous control, \(\U\), further extract \(\X_k\subseteq\X_j\), \(j=1,2,\ldots,n\), until the desired properties (\'a la Definition \ref{d:partition}) are met. It is also useful to know under what conditions on the subsystems dynamics, \((A_i,b_i)\), \(i\in\P\) and the set of admissible continuous control, \(\U\), a state-space abstraction is admitted by the switched system \eqref{e:swsys}. The following set of results is dedicated to this purpose.
 
    \begin{proposition}
    \label{prop:auxres1}
        Suppose that the following conditions hold:
        \begin{enumerate}[label=\alph*),leftmargin=*]
            \item \(\X = \R^d\),
            \item \(N = d\),
            \item all \(A_i\), \(i\in\P\) are diagonal matrices with non-zero diagonal elements,
            \item all \(b_i\), \(i\in\P\) have only the \(i\)-th element non-zero, and
            \item \(\U = \R\).
        \end{enumerate}
        Then \(\bigl(\X_j\bigr)_{j=0}^{d+1}\), where \(\X_0 = \{0_d\}\), \(\X_j = \bigl\{x\in\R^d\:|\:x_j\neq 0\:\text{and}\:x_i=0,\:i=1,2,\ldots,d,\:i\neq j\bigr\}\), \(\displaystyle{\X_{d+1} = \X\setminus\bigcup_{j=0}^{d}\X_j}\), is a state-space abstraction of the switched system \eqref{e:swsys}.
    \end{proposition}
    
     \begin{proposition}
    \label{prop:auxres2}
        Suppose that the following conditions hold:
        \begin{enumerate}[label=\alph*),leftmargin=*]
            \item \(\X = \R^d\),
            \item \(N = d\),
            \item all \(A_i\), \(i\in\P\) are anti-diagonal matrices with non-zero anti-diagonal elements,
            \item all \(b_i\), \(i\in\P\) have only the \((d-i+1)\)-th element non-zero, and
            \item \(\U = \R\).
        \end{enumerate}
        Then \(\bigl(\X_j\bigr)_{j=0}^{d+1}\), where \(\X_0 = \{0_d\}\), \(\X_j = \bigl\{x\in\R^d\:|\:x_j\neq 0\:\text{and}\:x_i=0,\:i=1,2,\ldots,d,\:i\neq j\bigr\}\), \(\displaystyle{\X_{d+1} = \X\setminus\bigcup_{j=0}^{d}\X_j}\), is a state-space abstraction of the switched system \eqref{e:swsys}.
    \end{proposition}

    Propositions \ref{prop:auxres1} and \ref{prop:auxres2} provide sufficient conditions on the set of states (condition a)), the number of subsystems (condition b)), the structure of the subsystems dynamics (conditions c) and d)), and the set of admissible continuous control, \(\U\), under which a switched system \eqref{e:swsys} admits a state-space abstraction (\'a la Definition \ref{d:partition}). In particular, we utilize diagonal (resp., anti-diagonal) structure of \(A_i\), \(i\in\Nset\) and sparse structure of \(b_i\), \(i\in\Nset\). Proofs of Propositions \ref{prop:auxres1} and \ref{prop:auxres2} are presented in Section \ref{s:all-proofs}. We note that Propositions \ref{prop:auxres1} and \ref{prop:auxres2} are only sufficient and present numerical examples beyond these assumptions in Section \ref{s:num-ex}.
    
    \begin{remark}
    \label{rem:non-switched}
        In the context of discrete-time linear time-invariant (LTI) systems, \(x(t+1) = Ax(t)+b\mu(t)\), \(t=0,1,\ldots,T-1\), design of maximum hands-off continuous control, \(\mu_T\), is achieved by employing \(\ell_1\)-norm approximation of the optimization problem \eqref{e:optprob1} and relying on the restricted isometry property of the transition matrix, \(\displaystyle{\Phi_{T}: = \pmat{\displaystyle{\prod_{t=1}^{T-1}}Ab & \cdots & Ab & b}}\) and existence of a unique solution to \(\Phi_{T}\mu_T = -A^T\xi\) \cite{Nagahara2016b}. The state-space abstraction based technique for the design of maximum hands-off control, presented in this paper, also adds to the literature on maximum hands-off control of discrete-time LTI systems. Indeed, a discrete-time LTI system is a special case of the switched system \eqref{e:swsys} with \(N=1\), \(A_i=A\), \(b_i=b\), \(i\in\Nset\). 
    \end{remark}
    

    We will now present numerical examples.
\section{Numerical examples}
\label{s:num-ex}
    \begin{example}
    \label{ex:num-ex1}
    Consider the switched system \eqref{e:swsys}. Let \(\P = \{1,2\}\) with\\ \((A_1,b_1) = \Bigl(\pmat{1 & 0\\0 & 1},\pmat{1\\0}\Bigr)\), \((A_2,b_2) = \Bigl(\pmat{0 & 1\\1 & 0},\pmat{0\\1}\Bigr)\).   
    Let \(\X = \biggl\{\pmat{x_1\\x_2}\:|\:\pmat{-10\\-10}\leq\pmat{x_1\\x_2}\leq\pmat{10\\10}\biggr\}\),\(\E(\Nset) = \{(1,2),(2,1)\}\), \(\U = [-10,10]\), \(\xi = \pmat{0\\10}\) and \(T=5\) units of time.
    
    First, we note that the switched system under consideration admits a state-space abstraction as described in Definition \ref{d:partition}. We have 
    \(\X_0 = \{0_2\}\), \(\X_1 = \biggl\{\pmat{x_1\\0}\:|\:-10\leq x_1\leq 10\biggr\}\), \(\X_2 = \biggl\{\pmat{0\\x_2}\:|\:-10\leq x_2\leq 10\biggr\}\), \(\X_3 = \biggl\{\pmat{x_1\\x_2}\:|\:-10\leq x_1\leq 10,-10\leq x_2\leq 10,x_1\neq 0,x_2\neq 0\biggr\}\). For \(\X_j = \X_1\) and \(i=1\), \(\X_k=\X_1\) and \(\X_\ell=\X_0\) with \(\mu=-x_1\), for \(\X_j = \X_1\) and \(i=2\), \(\X_k=\X_2\) and \(\X_\ell=\X_0\) with \(\mu=-x_1\),
    for \(\X_j = \X_2\) and \(i=1\), \(\X_k=\X_2\) and \(\X_\ell=\X_3\) with any \(\mu\in\U\setminus\{0\}\),
    for \(\X_j = \X_2\) and \(i=2\), \(\X_k=\X_1\) and \(\X_\ell=\X_3\) with any \(\mu\in\U\setminus\{0\}\),
    for \(\X_j = \X_3\) and \(i=1\), \(\X_k=\X_3\) and \(\X_\ell=\X_3\) with any \(\mu\in\U\setminus\{0\}\), and
    for \(\X_j = \X_3\) and \(i=2\), \(\X_k=\X_3\) and \(\X_\ell=\X_3\) with any \(\mu\in\U\setminus\{0\}\).
    
    Second, we construct \(G(V,E)\) that contains \(V = \{v_0,v_1,v_2,v_3\}\) with \(\ell(v_0)=\X_0\), \(\ell(v_1)=\X_1\), \(\ell(v_2)=\X_2\), \(\ell(v_3)=\X_3\), \(E = E_1\cup E_2\), where \(E_1 = \{(v_1,v_1),(v_1,v_2),(v_2,v_2),(v_2,v_1)\),\\\((v_3,v_3),(v_3,v_3)\}\) with \(\ell(v_1,v_1) = (1,0)\), \(\ell(v_1,v_2) = (2,0)\), \(\ell(v_2,v_2) = (1,0)\), \(\ell(v_2,v_1) = (2,0)\), \(\ell(v_3,v_3) = (1,0)\), \(\ell(v_3,v_3) = (2,0)\), and 
    \(E_2 = \{(v_1,v_0),(v_1,v_0),(v_2,v_3),(v_2,v_3)\),\\\((v_3,v_3),(v_3,v_3)\}\) with \(\ell(v_1,v_0) = (1,\mu=-x_1)\), \(\ell(v_1,v_0) = (2,\mu=-x_1)\), \(\ell(v_2,v_3) = (1,\mu\neq 0)\), \(\ell(v_2,v_3) = (2,\mu\neq 0)\), \(\ell(v_3,v_3) = (1,\mu\neq 0)\), \(\ell(v_3,v_3) = (2,\mu\neq 0)\).
    
    Third, \(\xi\in\X_s=\X_2\). We construct the set \(\W_\T(\X_s,\X_0) = \{W_5^{(1)} = v_2,(v_2,v_2),v_2,(v_2,v_1)\),\\\(v_1,(v_1,v_1),v_1,(v_1,v_0),v_0\), \(W_5^{(2)}=v_2,(v_2,v_1),v_1,(v_1,v_0),v_0\}\). We have \(w(W_5^{(1)})=4\) and \(w(W_5^{(2)})=2\). Thus, \(W_5^{(2)}\) qualifies to be a \(\xi\)-hands-off walk.
    
    Fourth, we obtain \((\nu^*_T,\mu^*_T) = \Biggl(\pmat{2\\1\\2\\1\\2},\pmat{0\\-10\\0\\0\\0}\Biggr)\). It follows that \(x(T)=0_2\) and \(\norm{\Delta(\nu_T)}_0+\norm{\mu_T}_0=4+1=5\). We have that \(x(0)=\xi\), \(x(1)=\pmat{10\\0}\), \(x(2)=x(3)=x(4)=x(5)=\pmat{0\\0}\).   
    \end{example}
    
    \begin{example}
    \label{ex:num-ex2}
        Consider the switched system \eqref{e:swsys}. Let \(\Nset = \{1,2\}\) with\\ \((A_1,b_1) = \pmat{\pmat{1 & 2\\2 & 4},\pmat{1\\2}}\), \((A_2,b_2) = \pmat{\pmat{1 & 2\\3 & 4},\pmat{1\\1}}\). Let \(\X = \{x\in\R^2\:|\:x_i\geq 0,\:i=1,2\}\), \(E(\Nset) = \{(1,1),(1,2),(2,1),(2,2)\}\), \(\U=\R\), \(\xi = \pmat{0\\10}\) and \(T=5\) units of time.
        
        First, the switched system under consideration admits a state-space abstraction as described in Definition \ref{d:partition}. We have 
    \(\X_0 = \{0_2\}\), \(\X_1 = \biggl\{\pmat{x_1\\0}\:|\:x_1>0\biggr\}\), \(\X_2 = \biggl\{\pmat{0\\x_2}\:|\:x_2 > 0\biggr\}\), \(\X_3 = \biggl\{\pmat{x_1\\x_2}\:|\:x_1> 0, x_2> 0\biggr\}\). For 
    \(\X_j = \X_1\) and \(i=1\), \(\X_k=\X_3\) and \(\X_\ell=\X_0\) with \(\mu=-x_1\), 
    for \(\X_j = \X_1\) and \(i=2\), \(\X_k=\X_3\) and \(\X_\ell=\X_3\) with \(\mu=-x_1\),
    for \(\X_j = \X_2\) and \(i=1\), \(\X_k=\X_3\) and \(\X_\ell=\X_0\) with any \(\mu=-2x_2\),
    for \(\X_j = \X_2\) and \(i=2\), \(\X_k=\X_3\) and \(\X_\ell=\X_2\) with any \(\mu=-2x_2\),
    for \(\X_j = \X_3\) and \(i=1\), \(\X_k=\X_3\) and \(\X_\ell=\X_3\) with any \(\mu\in\U\setminus\{0\}\), and
    for \(\X_j = \X_3\) and \(i=2\), \(\X_k=\X_3\) and \(\X_\ell=\X_3\) with any \(\mu\in\U\setminus\{0\}\).

        Second, we construct \(G(V,E)\) that contains \(V = \{v_0,v_1,v_2,v_3\}\) with \(\ell(v_0)=\X_0\), \(\ell(v_1)=\X_1\), \(\ell(v_2)=\X_2\), \(\ell(v_3)=\X_3\), \(E = E_1\cup E_2\), where \(E_1 = \{(v_1,v_3),(v_1,v_3),(v_2,v_3),(v_2,v_3)\),\\\((v_3,v_3),(v_3,v_3)\}\) with \(\ell(v_1,v_3) = (1,0)\), \(\ell(v_1,v_3) = (2,0)\), \(\ell(v_2,v_3) = (1,0)\), \(\ell(v_2,v_3) = (2,0)\), \(\ell(v_3,v_3) = (1,0)\), \(\ell(v_3,v_3) = (2,0)\), and 
    \(E_2 = \{(v_1,v_0),(v_1,v_2),(v_2,v_0),(v_2,v_2)\),\\\((v_3,v_3),(v_3,v_3)\}\) with \(\ell(v_1,v_0) = (1,\mu=-x_1)\), \(\ell(v_1,v_2) = (2,\mu=-x_1)\), \(\ell(v_2,v_0) = (1,\mu=-2x_2)\), \(\ell(v_2,v_2) = (2,\mu=-2x_2)\), \(\ell(v_3,v_3) = (1,\mu\neq 0)\), \(\ell(v_3,v_3) = (2,\mu\neq 0)\).
    
        Third, \(\xi\in\X_s = \X_2\). Notice that there is a direct edge from \(v_2\) to \(v_0\) with label \((1,-2x_2)\) but no path through any other vertex. Thus, \(W_T = v_2,(v_2,v_0),v_0\) qualifies to be a \(\xi\)-hands-off-walk/ 
        
        Fourth, we obtain \((\nu^*_T,\mu^*_T) = \pmat{\pmat{1\\1\\1\\1\\1},\pmat{-20\\0\\0\\0\\0}}\). It follows that \(x(T) = 0_2\) and \(\norm{\Delta(\nu_T)}_{0}+\norm{\mu_T}_{0} = 1+1 = 2\). We have \(x(0)=\xi\), \(x(1)=x(2)=x(3)=x(4) = \pmat{0\\0}\).
    \end{example}
    
    \begin{remark}
    \label{rem:not-hold}
    \rm{
        In the absence of restrictions on discrete and continuous control actions, maximum hands-off hybrid control for the switched systems in Examples \ref{ex:num-ex1} and \ref{ex:num-ex2} can be attempted to be designed using a convex relaxation of the optimization problem \eqref{e:optprob1} \cite{abc}. In particular, an \(\ell_1\)-norm approximation can be considered. However, for a solution of the \(\ell_1\)-optimization problem to match a solution of the \(\ell_0\)-optimization problem, we need to rely on restricted isometry property (resp., range space property) of the transition matrix, \(\displaystyle{\Phi_{\nu_T}: = \pmat{\displaystyle{\prod_{t=1}^{T-1}}A_{\nu(t)}b_{\nu(0)} & \cdots & A_{\nu(T-1)}b_{\nu(T-2)} & b_{\nu(T-1)}}}\) and existence of a unique solution to \(\Phi_{\nu_T}\mu_T = -\displaystyle{\prod_{t=0}^{T-1}}A_{\nu(t)}\xi\), as demonstrated in \cite{abc}. In general, verifying the above conditions numerically is not a straightforward task. We typically rely on sufficient conditions on the elements of the matrices, \((A_i,b_i)\), \(i\in\Nset\) for this purpose, see e.g., \cite[Propositions 3 and 4]{abc}. We note that in the setting of Examples \ref{ex:num-ex1} and \ref{ex:num-ex2}, \cite[Propositions 3 and 4]{abc} do not hold. 
        We have, however, been able to solve the optimization problem \eqref{e:optprob1} for these examples using the techniques proposed in this paper. 
    }
    \end{remark}
\section{Concluding remarks}
\label{s:concln}
    In this paper we studied an algorithm for the design of maximum hands-off hybrid control sequence for discrete-time switched linear systems under pre-specified restrictions on admissible discrete and continuous control. We employed graph theory as the primary apparatus for our analysis. Our results are more general than the existing literature but rely on the existence of a finite state-space abstraction for the switched system, which is restrictive. The problem of computation of maximum hands-off control sequences for switched systems using graph algorithms beyond the admissibility of a state-space abstraction, described in Definition \ref{d:partition}, is currently under investigation and the findings will be reported elsewhere. 

\section{Proofs of our results}
\label{s:all-proofs}
    \begin{proof}[Proof of Theorem \ref{t:mainres}]
        Let \((\overline{\nu}_T,\overline{\mu}_T)\) be a hybrid control sequence obtained from Algorithm \ref{algo:design}. We need to show that it is a maximum hands-off hybrid control sequence.
    
    Firstly, let \(W^*_T=\overline{v}_0,(\overline{v}_0,\overline{v}_1),\overline{v}_1,\ldots,\overline{v}_{r-1},(\overline{v}_{r-1},\overline{v}_{r}),\overline{v}_r\in
    \W_\T(\X_s,\X_0)\) be the \(\xi\)-hands-off walk employed in the design of \((\overline{\nu}_T,\overline{\mu}_T)\). By construction of the transition graph, \(G(V,{E})\), for the switched system \eqref{e:swsys}, we have that for any two consecutive edges, \((v',v'')\) and \((v'',v''')\) that appear on \(W^*_T\), \((\alpha(v',v''),\alpha(v'',v'''))\in\E(\Nset)\), and for any edge, \((v',v'')\) that appears on \(W^*_T\), \(\beta(v',v'')\in\{0,u\}\), where \(u\in\U\). In view of \ref{step:algo_step2} of Algorithm \ref{algo:design}, the hybrid control sequence, \((\overline{\nu}_T,\overline{\mu}_T)\in\H\).
    
    Secondly, by definition of state-space abstraction and its corresponding transition graph, we have that all \(x\in\ell(\overline{v}_k)\) are steered to \(y=A_ix+B_iu\in\ell(\overline{v}_{k+1})\), where \(\alpha(\overline{v}_k,\overline{v}_{k+1})=i\) and \(\beta(\overline{v}_k,\overline{v}_{k+1})=u\), \(k=0,1,\ldots,r-1\). Further, by definition of a \(\xi\)-hands-off walk on \(G(V,{E})\), \(\xi\in\ell(\overline{v}_0)\) and \(\ell(\overline{v}_r)=\X_0=\{0_d\}\). Also, \(A_i0_d+b_i0=0_d\) for all \(i\in\P\). It follows that \(x(0)=\xi\) is steered to \(x(T)=0_d\) under \((\overline{\nu}_T,\overline{\mu}_T)\). Thus, the hybrid control sequence, \((\overline{\nu}_T,\overline{\mu}_T)\in\C_{\xi,T}\).
    
    Thirdly, assume that \((\overline{\nu}_T,\overline{\mu}_T)\) is not a maximum hands-off hybrid control sequence. Then there must exist \((\tilde{\nu}_T,\tilde{\mu}_T)\in\C_{\xi,T}\) that satisfies\footnote{Notice that for \(r<T\), a constant term with a maximum value \(T-r-2\) gets added on both sides of \eqref{e:pf1_step1}.} 
    \begin{align}
    \label{e:pf1_step1}
        \norm{\Delta(\tilde{\nu}_T)}_{0}+\norm{\Delta(\tilde{\mu}_T)}_{0}<\norm{\Delta(\overline{\nu}_T)}_{0}+\norm{\Delta(\overline{\mu}_T)}_{0}.
    \end{align} 
    Now, 
    \begin{align*}
        \norm{\Delta(\tilde{\nu}_T)}_0 &= \norm{\pmat{\tilde{\nu}(1)-\tilde{\nu}(0)\\\tilde{\nu}(2)-\tilde{\nu}(1)\\\vdots\\\tilde{\nu}(T-1)-\tilde{\nu}(T-2)}}_0
       =\norm{\pmat{\overline{\alpha}(\tilde{v}_0,\tilde{v}_1)\\\overline{\alpha}(\tilde{v}_1,\tilde{v}_2)\\\vdots\\\overline{\alpha}(\tilde{v}_{T-2},\tilde{v}_{T-1})}}_0\\
       \intertext{and}
        \norm{\tilde{\mu}_T}_0 &= \norm{\pmat{\tilde{\mu}(0)\\\tilde{\mu}(1)\\\vdots\\\tilde{\mu}(T-1)}}_0
       =\norm{\pmat{\overline{\beta}(\tilde{v}_0,\tilde{v}_1)\\\overline{\beta}(\tilde{v}_1,\tilde{v}_2)\\\vdots\\\overline{\beta}(\tilde{v}_{T-2},\tilde{v}_{T-1})}}_0,    \end{align*}
    where \(\tilde{W}_T = \tilde{v}_0,(\tilde{v}_0,\tilde{v}_1),\tilde{v}_1,\ldots,\tilde{v}_{r-1},(\tilde{v}_{r-1},\tilde{v}_{r}),\tilde{v}_r\) is the \(\xi\)-hands-off walk corresponding to \((\tilde{\nu}_T,\tilde{\mu}_T)\). Here, the correspondence is in the sense of construction of \((\tilde{\nu}_T,\tilde{\mu}_T)\) from \(\tilde{W}_T\) as in Algorithm \ref{algo:design}. In view of \eqref{e:pf1_step1}, we must have that 
    \[
        \norm{\pmat{\overline{\alpha}(\tilde{v}_0,\tilde{v}_1)\\\overline{\alpha}(\tilde{v}_1,\tilde{v}_2)\\\vdots\\
        \overline{\alpha}(\tilde{v}_{T-2},\tilde{v}_{T-1})}}_0 + \norm{\pmat{\overline{\beta}(\tilde{v}_0,\tilde{v}_1)\\\overline{\beta}(\tilde{v}_1,\tilde{v}_2)\\\vdots\\
        \overline{\beta}(\tilde{v}_{T-2},\tilde{v}_{T-1})}}_0 < \norm{\pmat{\overline{\alpha}(\overline{v}_0,\overline{v}_1)\\\overline{\alpha}(\overline{v}_1,\overline{v}_2)\\\vdots\\
        \overline{\alpha}(\overline{v}_{T-2},\overline{v}_{T-1})}}_0 + \norm{\pmat{\overline{\beta}(\overline{v}_0,\overline{v}_1)\\\overline{\beta}(\overline{v}_1,\overline{v}_2)\\\vdots\\
        \overline{\beta}(\overline{v}_{T-2},\overline{v}_{T-1})}}_0.
    \]
   Then it must hold that
    \[
        \sum_{m=1}^{r-1}\overline{\alpha}(\tilde{v}_m,\tilde{v}_{m+1}) + \sum_{m=1}^{r-1}\overline{\beta}(\tilde{v}_m,\tilde{v}_{m+1}) < \sum_{m=1}^{r-1}\overline{\alpha}(\overline{v}_m,\overline{v}_{m+1}) + \sum_{m=1}^{r-1}\overline{\beta}(\overline{v}_m,\overline{v}_{m+1}).
    \]
   Thus, \(w(W^*_T)>w(\tilde{W}_T)\) and \(W^*_T\) is not a \(\xi\)-hands-off walk, which is a contradiction. The hybrid control sequence, \((\overline{\nu}_T,\overline{\mu}_T)\), is, therefore, a maximum hands-off control sequence for the switched system \eqref{e:swsys}.
    
    The assertion of Theorem \ref{t:mainres} follows.

    \end{proof}
    
    The following notation will be useful in our proofs of Propositions \ref{prop:auxres1}-\ref{prop:auxres2}: Let \(A_i(p,q)\) denote the element of \(A_i\) at the \(p\)-th row and \(q\)-th column, \(p\in\{1,2,\ldots,d\}\), \(q\in\{1,2,\ldots,d\}\), \(i\in\Nset\). Let \(b_i(p)\) denote the element of \(b_i\) at the \(p\)-th row, \(p\in\{1,2,\ldots,d\}\), \(i\in\Nset\). We let \(x_i\) denote the \(i\)-th element of \(x\), \(i=1,2,\ldots,d\).
    
    \begin{proof}[Proof of Proposition \ref{prop:auxres1}]
        We need to show that \(\bigl(\X_j\bigr)_{j=0}^{d+1}\) satisfies conditions i)-v) in Definition \ref{d:partition}.
        
        First, in view of a), conditions i)-iii) are satisfied by the choice of the elements of \(\bigl(\X_j\bigr)_{j=0}^{d+1}\). It remains to show that conditions iv)-v) hold.
        
        Second, in view of c), we have that for all \(x\in\X_j\) and \(i\in\Nset\), \(y=A_ix\in\X_k\), \(j,k=1,2,\ldots,d+1\), \(j=k\). Thus, condition iv) holds.
        
        Third, in view of b), c), d) and e), we have that for all \(x\in\X_j\) and \(i=j\in\Nset\), \(y=A_ix+b_i\mu\in\X_0\) with \(\mu = -\frac{A_i(j,j)}{b_i(j)}x_j\), \(j=1,2,\ldots,d\), and for all \(x\in\X_j\) and \(i\neq j\in\Nset\), \(y=A_ix+b_i\mu\in\X_{d+1}\) with any \(\mu\in\U\setminus\{0\}\), \(j=1,2,\ldots,d\). Further, for all \(x\in\X_{d+1}\) and \(i\in\Nset\), \(y=A_ix+b_i\mu\in\X_{d+1}\) with any \(\mu\in\U\setminus\{0\}\).
        
        The assertion of Proposition \ref{prop:auxres1} follows.
    \end{proof}
    
    \begin{proof}[Proof of Proposition \ref{prop:auxres2}]
        We need to show that \(\bigl(\X_j\bigr)_{j=0}^{d+1}\) satisfies conditions i)-v) in Definition \ref{d:partition}.
        
        First, in view of a), conditions i)-iii) are satisfied by the choice of the elements of \(\bigl(\X_j\bigr)_{j=0}^{d+1}\). It remains to show that conditions iv)-v) hold.
        
        Second, in view of c), we have that for all \(x\in\X_j\) and \(i\in\Nset\), \(y=A_ix\in\X_{d-j+1}\), \(j=1,2,\ldots,d\). Further, for all \(x\in\X_{d+1}\) and \(i\in\Nset\), \(y=A_ix\in\X_{d+1}\). Thus, condition iv) holds.
        
        Third, in view of b), c), d) and e), we have that for all \(x\in\X_j\) and \(i=j\in\Nset\), \(y=A_ix+b_i\mu\in\X_0\) with \(\mu = -\frac{A_i(j,d-j+1)}{b(d-j+1)}x_j\), \(j=1,2,\ldots,d\), and for all \(x\in\X_j\) and \(i\neq j\in\Nset\), \(y=A_ix+b_i\mu\in\X_{d+1}\) with any \(\mu\in\U\setminus\{0\}\). Thus, condition v) holds.
        
        The assertion of Proposition \ref{prop:auxres2} follows.
    \end{proof}


\end{document}